\documentclass{amsart}
\usepackage{latexsym,amsxtra,amscd,ifthen}
\usepackage{amsfonts}
\usepackage{verbatim}
\usepackage{amsmath}
\usepackage{amsthm}
\usepackage{amssymb}
\usepackage{url}
\usepackage[arrow,matrix]{xy}

\allowdisplaybreaks[4]

\newtheorem{theorem}{Theorem}[section]

\newtheorem{lemma}[theorem]{Lemma}

\newtheorem{proposition}[theorem]{Proposition}
\newtheorem{prop-def}[theorem]{Proposition-Definition}
\theoremstyle{definition}
\newtheorem{definition}[theorem]{Definition}
\newtheorem{remark}[theorem]{Remark}
\newtheorem{question}[theorem]{Question}
\newtheorem{project}[theorem]{Project}
\newtheorem{conjecture}[theorem]{Conjecture}

\newtheorem{example}[theorem]{Example}

\numberwithin{equation}{section}

\newcommand{\inth}{\textstyle \int}

\DeclareMathOperator{\gldim}{gldim}

\DeclareMathOperator{\Ext}{Ext}

\DeclareMathOperator{\Aut}{Aut}

\DeclareMathOperator{\injdim}{injdim}

\DeclareMathOperator{\GKdim}{GKdim}

 \DeclareMathOperator{\im}{im}

\newcommand{\fm}{\mathfrak{m}}

\begin{document}

\title{Survey on Hopf algebras of GK-dimension 1 and 2}

\author{K.A. Brown}
\address{School of Mathematics and Statistics\\
University of Glasgow\\ Glasgow G12 8QW\\
Scotland}
\email{ken.brown@glasgow.ac.uk}

\author{J.J. Zhang}
\address{Department of Mathematics\\
Box 354350\\
University of Washington\\
Seattle, WA 98195, USA}
\email{zhang@math.washington.edu}
\begin{abstract}
This is a brief survey on the current state of play on the 
programs to classify infinite dimensional Hopf algebras of 
Gel'fand-Kirillov dimension one or two. We list a number of 
open questions and suggest directions for future work.
\end{abstract}





\maketitle

\setcounter{section}{-1}
\section{introduction}
\label{xxsec0}

During the last 20 years, there has been significant progress 
towards a deeper understanding of infinite dimensional Hopf 
algebras satisfying some natural finiteness conditions. There 
has been wide interest in this effort, most notably in the 
programs concerning (both finite and infinite dimensional) 
Nichols algebras led by Andruskiewitsch, Angiono, Heckenberger, 
Schneider and others \cite{AA06, AS00, AS04, AAH19, Ang13}, 
and the program concerning affine noetherian Hopf algebras 
initiated by Brown and Goodearl in 1997 
\cite{BG97, Br98, BZ08, BZ10, LWZ07, WZ03, WLD16}. The subject 
of the present short survey is the classification and analysis 
of infinite dimensional Hopf algebras of low Gel'fand-Kirillov 
dimension. This topic lies at the intersection of the above 
two projects. Thus, as well as yielding results which are of 
interest in their own right, the research which we describe is 
a rich source both of examples and of methods which will prove 
useful in further work on both the above long-term endeavors.

The low-dimensional classification project has necessitated the 
development and refinement over the last 20 years of a number 
of algebraic concepts and homological invariants, for example 
AS-regularity, the Nakayama automorphism, the integral 
order and the integral minor, some of which we shall review here. 
These have played important roles in the further study of infinite 
dimensional noetherian Hopf algebras, strengthening connections to 
noncommutative algebraic geometry and the representation theory 
of quantum groups, with further applications likely in combinatorics, 
representation theory, statistical mechanics, topology and 
mathematical physics. To give one example, the project to 
classify prime AS-Gorenstein Hopf algebras of GK-dimension 1 in 
characteristic 0 has culminated in work of Liu \cite{Li20} which 
provides large families of examples ripe for exploitation in the 
study of tensor categories, support varieties, Hochschild 
cohomology, and so on. On the other hand despite (or rather, 
because of) this recent progress, very many open questions remain, 
many of which we have listed and discussed in what follows. For 
these reasons it is therefore a good time to briefly survey these 
classification programs.

Note that this account does not touch upon the large volume of recent  
significant research concerning finite dimensional (semisimple or 
not) Hopf algebras, or Nichols algebras, or tensor categories. For 
recent surveys on these topics we refer the reader to \cite{And14}, 
\cite{And17}, \cite{EGNO15} respectively. 

\subsection{General setup}
\label{xxsec0.1}
Throughout, the base field will be denoted by $\Bbbk$ and will be 
assumed to be algebraically closed. For many of the results 
described this hypothesis can undoubtedly be weakened, but we will 
not pursue this question. On the other hand, the characteristic of 
the field will often be significant, so we will make clear the 
hypotheses concerning the characteristic of $\Bbbk$ on a case by 
case basis; whenever the characteristic is not mentioned, this 
means that the characteristic is arbitrary. All vector spaces, 
(co)algebras, tensor products, etc. are taken over $\Bbbk$. An 
algebra is called {\it affine} if it is finitely generated over 
$\Bbbk$. We refer to Montgomery's book \cite{Mo93} as a basic 
reference for Hopf algebras. For any Hopf algebra $H$, we denote 
the multiplication, unit map, comultiplication, counit and antipode 
by $m, u, \Delta, \epsilon$ and $S$ respectively. We use the same symbol 
$\Bbbk$ to denote both the trivial $\Bbbk$-algebra, and the 
trivial module over $H$, 
namely, $\Bbbk=H/\ker \epsilon$.  Where not otherwise indicated, 
modules are left modules. Let $H^{op}$ denote the opposite ring of 
$H$, so a right $H$-module can be viewed as a left $H^{op}$-module 
and an $H$-bimodule is identified with a left $H \otimes H^{op}$-module.
Many examples and open questions on the topics discussed here can be found 
in textbooks and in the survey papers \cite{Br98, Br07, BGi14, Go13}. 
In the present paper we will update these sources on their coverage of 
algebras of small GK-dimension, but we recommend them as references 
for wider issues concerning noetherian Hopf algebras.

\subsection{Gel'fand-Kirillov dimension}
\label{xxsec0.2}
Let $A$ be an algebra over $\Bbbk$. The 
{\it Gel'fand-Kirillov dimension} (or {\it GK-dimension} for short) 
of $A$ is defined to be
\begin{equation}
\notag
\GKdim A:=\sup_{V} \limsup_{n\to\infty} \left(\log_{n} 
(\dim_{\Bbbk} V^n)\right)
\end{equation}
where $V$ runs over all finite dimensional subspaces of $A$. We refer 
to \cite{KL00} for basic properties of GK-dimension of algebras and 
modules. Note that finite dimensional algebras have GK-dimension zero. 
Conversely, every algebra of non-zero GK-dimension is clearly infinite 
dimensional. Examples of Borho-Kraft and Warfield \cite[Theorem 2.9]{KL00},
together with the Bergman Gap Theorem \cite[Theorem 2.5]{KL00} show that 
GK-dimension of an affine algebra can take any value from the set
\begin{equation}\label{E0.0.1}\tag{E0.0.1} 
\{0\} \cup \{1\} \cup [2,\infty].
\end{equation}
The trichotomy \eqref{E0.0.1} motivates the study of Hopf $\Bbbk$-algebras 
of GK-dimension 1 and 2, since these classes constitute a laboratory where 
classification is feasible and where more general conjectures can be tested. 
Note, however, that there is currently no known example of a Hopf algebra
with finite non-integral GK-dimension. Thus we repeat 
\cite[Question F]{BGi14} (which was also asked by Zhuang \cite{Zhu13}).

\begin{question}\cite[Question F]{BGi14}
\label{xxque0.1}
Does every Hopf algebra have either infinite or integral GK-dimension?
\end{question}

\noindent For further general questions on the GK-dimension of 
Hopf algebras, see \cite{BGi14}.

\subsection{Artin-Schelter and Cohen-Macaulay properties}
\label{xxsec0.3}
The following concepts are key to the analysis of the Hopf algebras studied 
in this paper. Recall that a Hopf $\Bbbk$-algebra $H$ is {\it Artin-Schelter 
Gorenstein} ({\it AS-Gorenstein} for short) if
\begin{enumerate}
\item[(AS1)]
$\injdim \,_H{H} = d <\infty$,
\item[(AS2)] 
$\dim_{\Bbbk} \Ext^d_{H}(_H{\Bbbk}, _H H) = 1$, 
$\Ext^i_{H}(_H{\Bbbk}, _H H) = 0$ for all $i\neq  d$,
\item[(AS3)] 
the right $H$-module versions of the conditions (AS1), (AS2) hold.
\end{enumerate}
We say $H$ is {\it Artin-Schelter regular} ({\it AS-regular}
for short) if it is AS-Gorenstein and has finite global dimension; 
in this case $\gldim H = d$. The \emph{homological grade} of a 
non-zero $H$-module $M$ is defined to be 
$j(M) := \mathrm{min}\{j : \mathrm{Ext}^j_H(M,H) \neq 0\} \cup 
\{\infty \}$. Suppose that $H$ has finite GK-dimension. Then $H$ 
is \emph{GK-Cohen-Macaulay} if, for all non-zero finitely 
generated $H$-modules $M$,
$$ j(M) + \mathrm{GKdim} M = \mathrm{GKdim} H. $$
Notice that if $H$ is AS-Gorenstein and GK-Cohen-Macaulay then
$$ \injdim H =  \mathrm{GKdim} H. $$ 

\noindent All known noetherian Hopf algebras are 
AS-Gorenstein \cite{BZ08}. A well-known conjecture 
in the study of noetherian Hopf algebras is the following, 
which was posted in \cite[Question E]{Br07} and 
\cite[Question 3.5]{Go13}.

\begin{conjecture}[Brown-Goodearl Conjecture]
\label{xxcor0.2}
Every noetherian Hopf algebra is AS-Gorenstein.
\end{conjecture}

\noindent 
We refer to \cite{Br98, Br07, BZ08, Go13} for further discussion 
of homological aspects of noetherian Hopf algebras.

\subsection{Organization}
\label{xxsec0.4}
The paper is organized as follows. In Section 1 we consider 
prime Hopf algebras of GK-dimension one, and in Section 2, 
we mainly consider Hopf algebra domains of GK-dimension two, 
with a digression to review Hopf Ore extensions. These 
two sections can be viewed as an expanded version of
\cite[Section 4]{Go13}. Section 3 contains a review and brief 
commentary on the many questions stated in Sections 0, 1 and 2.

\section{Hopf algebras of GK-dimension 1}
\label{xxsec1}

In this section we study Hopf algebras which are prime of 
GK-dimension one. By a famous result of Small and Warfield 
\cite{SW84}, when such an algebra (whether Hopf or not) is 
affine it is a finite module over its affine center, so it 
is noetherian and satisfies a polynomial identity. Thus if 
$H$ is a prime affine Hopf algebra of GK-dimension 1, it is 
AS-Gorenstein and GK-Cohen-Macaulay by 
\cite[Theorems 0.1 and 0.2]{WZ03}. In particular, the injective 
dimension of $H$ (and its global dimension if $H$ is regular) 
are both 1.

\subsection{Preliminaries and key examples}
\label{xxsec1.1}

Recall the classical result that the only connected algebraic 
groups of dimension one over an algebraically closed field 
$\Bbbk$ are the additive and multiplicative groups of $\Bbbk$, 
see for example, \cite[Theorem 20.5]{Hu75} or \cite[Theorem III.10.9]{Bo91}. 
Translating to the language of Hopf algebras, the only affine 
commutative Hopf $\Bbbk$-algebra domains of GK-dimension 1 
are $\Bbbk [X]$ and $\Bbbk [X^{\pm 1}]$. But Tsen's theorem 
\cite[p. 374]{Co77} permits us to omit the commutative hypothesis, 
so combining \cite[Corollary 7.8 (a)$\Rightarrow$(b)]{LWZ07}
with \cite[Proposition 2.1]{GZ10} yields the following.

\begin{lemma} \cite[Proposition 2.1(a,b)]{GZ10}, 
\cite[Lemma 5.1]{BZ20}
\label{xxlem1.1} 
The only affine Hopf $\Bbbk$-algebra domains of GK-dimension 
one are the coordinate rings $\Bbbk [X]$ and $\Bbbk [X^{\pm 1}]$ 
of $(\Bbbk, +)$ and $(\Bbbk^{\times}, \times)$ respectively. 
\end{lemma}

\begin{remark}
\label{xxrem1.2}
In fact one can at least partly remove from Lemma \ref{xxlem1.1} 
the hypothesis that $H$ is affine: Tsen's theorem still ensures 
that the algebras involved are commutative, so $S^2 = \mathrm{Id}$ 
by \cite[Corollary 1.5.12]{Mo93} and hence $H$ is a union of 
affine Hopf subalgebras. Let $H$ be a non-affine Hopf 
$\Bbbk$-algebra domain of GK-dimension one. The following 
statements follow from Lemma \ref{xxlem1.1}. 
\begin{enumerate}
\item[(1)]
\cite[Proposition  2.1(c)]{GZ10}
Assume that $\Bbbk$ has characteristic 0. Then $H$ is a
group algebra $\Bbbk G$ where $G$ is a non-cyclic torsion-free
abelian group of rank one, i.e., a non-cyclic subgroup of 
$\mathbb{Q}$.
\item[(2)]
When $\Bbbk$ has positive characteristic and $H$ contains a 
nontrivial grouplike element, then $H$ is a group algebra $\Bbbk G$ 
where $G$ is a non-cyclic torsion-free abelian group of rank one.
\item[(3)]
If $H$ does not contain a nontrivial grouplike element, then 
$\Bbbk$ has positive characteristic, say $p$. If further $H$ 
is countably generated, then one can check that 
$$ H= \Bbbk \langle X_i : \, [X_i,X_j] = 0, \, 
X_i=F_i(X_{i+1}), \, (i,j \geq 1) \rangle,$$
with $X_i$ primitive for all $i$, where $\{F_i(t)\}_{i\geq 1}$
is a sequence of polynomials of the form
$$F_i(t)=\sum_{s\geq 0} \alpha_{i,s} t^{p^s}$$
for some $\alpha_{i,s}\in \Bbbk$.
\end{enumerate}
\end{remark}

\noindent The following issue thus remains to be clarified.

\begin{question} 
\label{xxque1.3} 
Let $\Bbbk$ have positive characteristic, and let $H$ be an 
uncountably generated Hopf domain over $\Bbbk$ of GK-dimension 
one. What is the structure of $H$?
\end{question}

For prime Hopf algebras of GK-dimension one which are not domains, 
the story is much more complicated and as yet incomplete even in 
characteristic 0. We begin by recalling some key examples.

\begin{example}
\label{xxex1.4} 
Let ${\mathbb D}$ be the infinite Dihedral group 
$\langle g,x\mid g^2=1, gxg=x^{-1}\rangle$ and let $H$ be the group 
algebra $\Bbbk {\mathbb D}$. Since ${\mathbb D}$ has no nontrivial 
finite normal subgroups, $H$ is prime by Connell's theorem 
\cite[Theorem 4.2.10]{Pas77}, and $H$ is regular if and only if 
$\Bbbk$ does not have characteristic 2, see \cite[Theorem 10.3.13]{Pas77}.
\end{example}

The second family is also well-known, 
sometimes called the {\it infinite Taft algebras}. 

\begin{example} 
\cite{Ta71}, \cite[Examples 2.7 and 7.3]{LWZ07}
\label{xxex1.5}
Let $n$ and $t$ be positive integers with $n \geq 2$ and 
$1 \leq t \leq n$, and let $\xi\in \Bbbk$ be a primitive $n$th 
root of 1. Let $H := H(n,t,\xi)$ be the $\Bbbk$-algebra 
generated by $x$ and $g$ subject to the relations
$$g^n = 1, \quad {\text{and}} \quad xg = \xi gx,$$
so $H$ is not commutative. The coalgebra structure of $H$ 
is defined by
$$\Delta(g)= g\otimes g, \qquad \epsilon(g) = 1$$ 
and 
$$\Delta (x) = x \otimes 1 + g^t \otimes x, \qquad \epsilon(x)=0.$$
Thus $H$ is cocommutative if and only if $t = n$. The antipode $S$ 
of $H$ is defined by
$$S(g) = g^{-1}, \quad {\text{and}} \quad
S(x) = -g^{-t}x = -\xi^{t}xg^{-t}.$$
Then $H$ is affine prime AS-regular of GK-dimension one, with 
center $Z(H) = \Bbbk [x^n]$.
\end{example}

In 2009 Liu discovered a new family of Hopf algebras of 
GK-dimension one: 

\begin{example} \cite[Proposition 2.1]{Li09}
\label{xxex1.6}
Let $n\geq 2$ and $\xi\in \Bbbk$ be a primitive $n$th root of 1.
Let $A(n,\xi)$ be generated by $x, g$ and the inverse $g^{-1}$, 
with relations
$$xg = \xi gx, \quad {\text{and}} \quad x^n = 1- g^n.$$
The coalgebra structure of $A(n,\xi)$ is defined by
$$\Delta(g)= g\otimes g, \qquad \epsilon(g) = 1$$ 
and 
$$\Delta (x) = x \otimes 1 + g \otimes x, \qquad \epsilon(x)=0.$$
Then $A(n,\xi)$ is affine prime AS-regular of GK-dimension one 
with center $\Bbbk [x^n]$.
\end{example}

\subsection{Integral order and minor}\label{xxsec1.2}
The project to classify prime regular Hopf algebras was begun 
in \cite{LWZ07}. The  homological integral was defined 
there - this has proved crucial for further analysis of 
AS-Gorenstein Hopf algebras, including classification 
projects. In this subsection we review the basic definitions, 
generalising the classical integrals of finite dimensional Hopf 
algebras \cite[Chapter 2]{Mo93}.

\begin{definition}\cite[Definition 1.1]{LWZ07}
\label{xxdef1.7}
Let $H$ be an AS-Gorenstein Hopf algebra of injective dimension 
$d$. Any nonzero element in $\Ext^d_{H}(_H{\Bbbk}, _H H)$ is 
called a {\it left homological integral} of $H$. Write 
$\inth^l := \Ext^d_{H}(_H{\Bbbk}, _H H)$. A nonzero element in 
$\inth^r  := \Ext^d_{H^{op}} ({\Bbbk}_{H},H_{H})$ is called a 
{\it right homological integral} of $H$. By abuse of language 
we also call $\inth^l$ and $\inth^r$ the {\it left} and 
{\it right homological integrals} of $H$ respectively.
\end{definition}

Notice that $\inth^l$ is an $H$-bimodule of $\Bbbk$-dimension 1, 
trivial as a left $H$-module but in general nontrivial on the 
right. For any 1-dimensional left (or right) $H$-module $M$,
we can write it as $H/I$ for an ideal $I$ with 
$\mathrm{dim}_{\Bbbk}H/I = 1$, or equivalently view it as an algebra
homomorphism $\phi_{M}: H\to H/I$. Then $\phi := \phi_{M}$
is a grouplike element in the dual Hopf algebra $H^{\circ}$, whose 
order is denoted by $o(M)\in {\mathbb{N}}\cup\{\infty\}$. 
Equivalently, $o(M)$ is the least positive integer $m$ such that 
$M^{\otimes m}$ is the trivial module $\Bbbk$; so $o(M) = 1$ if and 
only if $M = \Bbbk$.

Continuing with the above notation, the {\it left winding
automorphism} of $H$ associated to $\phi$ is defined
to be 
$$\Xi^l_{\phi}: h\to \sum \phi(h_1) h_2, \quad \forall \; h\in H,$$
and the {\it right winding automorphism} associated to $\phi$ is  
$$\Xi^r_{\phi}: h\to \sum h_1 \phi(h_2), \quad \forall \; h\in H.$$
Both $\Xi^l_{\phi}$ and $\Xi^r_{\phi}$ are in the group 
$\Aut_{alg}(H)$ of algebra automorphisms of $H$. Let $G^l_{\phi}$ 
[resp. $G^r_{\phi}$] (or $G^l_M$ [resp. $G^r_M$]) denote the 
subgroups of $\Aut_{alg}(H)$ generated by 
$\Xi^l_{\phi}$ [resp. $\Xi^r_{\phi}$]. 

\begin{definition}
\label{xxdef1.8}
Let $H$ be an AS-Gorenstein Hopf algebra.
\begin{enumerate}
\item[(1)]
\cite[Definition 2.2]{LWZ07}
The {\it integral order} of $H$ is  
$$io(H):=o(\inth^{l}) = |G^l_{\int^l}| \in {\mathbb{N}}\cup\{\infty\},$$
where $\inth^l$ is considered as a right $H$-module.
\item[(2)] 
$H$ is called {\it unimodular} if $io(H)=1$. 
\item[(3)] \cite[Definition 2.4]{BZ10} Suppose $io(H)< \infty$. The 
{\it integral minor} of $H$ is 
$$im(H):=|G^l_{\int^l}/(G^l_{\int^l}\cap G^r_{\int^l})|.$$
where the intersection is taken inside $\Aut_{alg}(H)$.
\end{enumerate}
\end{definition}

By \cite[Lemma 2.1]{LWZ07}, $io(H)$ equals $o(\inth^r)$, viewing 
$\inth^r$ as a left $H$-module. So it is not hard to show \cite{LWZ07, BZ10} 
that
$$ io(H) = |G^r_{\int^l}| = |G^l_{\int^r}| = |G^r_{\int^r}|.$$
Clearly $im(H)$ divides $io(H)$ if $io(H)$ is finite. For the Taft algebras
$H = H(n,t,\xi)$ of Example \ref{xxex1.5} it is not hard to check that 
$io(H) = n$ and $im(H) = n/\mathrm{gcd}(n,t)$, \cite[page 273]{BZ10}.

\subsection{Classification results}
\label{xxsec1.3}
The connection between the concepts just defined and classification 
was made clear in \cite{LWZ07}, with the following result. 

\begin{theorem} \cite[Theorem 7.1]{LWZ07}
\label{xxthm1.9}
Let $H$ be an affine prime AS-regular Hopf algebra of 
GK-dimension one. Then $io(H)<\infty$ and it equals the 
PI-degree of $H$. As a consequence, $H$ is unimodular if 
and only if $H$ is commutative. 
\end{theorem}

The commutative algebras which occur in Theorem \ref{xxthm1.9} 
are listed in Lemma \ref{xxlem1.1}. Progress was achieved 
towards a classification of the noncommutative ones in 
characteristic 0 by analysing their integral order and minor. 
Here are the two key theorems from \cite{BZ10}.

\begin{theorem} \cite[Theorem 4.1]{BZ10}
\label{xxthm1.10}
Assume that $\Bbbk$ has characteristic 0, and let $H$ be an 
affine prime AS-regular Hopf $\Bbbk$-algebra of 
GK-dimension one. Suppose $io(H)>1$ and $im(H)=1$. Then
$H$ is one of the following.
\begin{enumerate}
\item[(1)]
$H$ is the Taft algebra $H(n,n,\xi)$ of Example \ref{xxex1.5}. 
In this case $io(H)=n$.
\item[(2)]
$H$ is the Dihedral group algebra $\Bbbk {\mathbb D}$ of 
Example \ref{xxex1.4}. In this case $io(H)=2$.
\end{enumerate}
In both cases, $H$ is cocommutative.
\end{theorem}

The next result classifies all Hopf algebras in Theorem \ref{xxthm1.9} 
with $im(H)=io(H)$. To save space we will not give the definition
of the generalized Liu algebras, denoted by $B(n,w,\xi)$, see 
\cite[Section 3.4]{BZ10} for details. Here $\xi$ is a primitive $n$th 
root of 1 and $w$ is a positive integer. When $\mathrm{gcd}(n,w) =1$ 
one obtains precisely Liu's algebras from Example \ref{xxex1.6}.

\begin{theorem} \cite[Theorem 6.1]{BZ10}
\label{xxthm1.11}
Assume that $\Bbbk$ has characteristic 0, and let $H$ be an affine 
prime AS-regular Hopf algebra of 
GK-dimension one. Suppose $io(H)=im(H)>1$. Then
$H$ is one of the following.
\begin{enumerate}
\item[(1)] The Taft algebra $H(n,1,\xi)$ of 
Example \ref{xxex1.5}. 
\item[(2)] The generalized Liu algebra $B(n,w,\xi)$. 
\end{enumerate}
In both cases, $io(H)=n$ and $H$ is not cocommutative.
\end{theorem}

In \cite[Question 7.1]{BZ10} we asked if, in characteristic 0, every 
Hopf algebra in Theorem \ref{xxthm1.9} satisfies either $im(H)=1$ or 
$im(H)=io(H)$. The remarkable construction and result of Wu, Liu and 
Ding \cite{WLD16} answers this question in the negative, and in doing 
so provides the complete classification of the Hopf algebras in 
Theorem \ref{xxthm1.9} in the characteristic 0 case. The new class of 
Hopf algebras $D(m,d,\xi)$ is defined in \cite[Subsection 4.1]{WLD16}, 
requiring a couple of pages, and the proof that $D(m,d,\xi)$ is a Hopf 
algebra occupies 10 pages of \cite[Subsection 4.2]{WLD16}. Basic 
properties of $D(m,d,\xi)$ are described in \cite[Subsection 4.3]{WLD16}; 
in particular $io(D(m,d,\xi))=2m$ and $\im(D(m,d,\xi))=m$ 
\cite[Lemma 4.5]{WLD16}.

\begin{theorem} \cite[Theorem 8.3]{WLD16} 
\label{xxthm1.12} 
Asume that $\Bbbk$ has characteristic 0, and let $H$ be an affine 
prime AS-regular Hopf algebra of GK-dimension one. Suppose 
$io(H)\neq im(H)\neq 1$. Then $H$ is isomorphic to $D(m,d,\xi)$ for 
some choice of $m,d,\xi$.
\end{theorem}

From Theorems \ref{xxthm1.9}, \ref{xxthm1.10}, \ref{xxthm1.11} and 
\ref{xxthm1.12} one sees that, when $\Bbbk$ has characteristic 0 and 
the Hopf $\Bbbk$-algebra $H$ is prime affine AS-regular of GK-dimension 
one, 
$${\text{$io(H)/im(H)$ is either 1, 2 or  $io(H)$.}}$$ 
By \cite[Proposition 4.9]{WLD16}, $D(m,d,\xi)$ is not pointed, giving a 
negative answer to \cite[Question 2.16]{Go13} and \cite[Section 0.5]{BZ10}. 

\subsection{Further questions and comments on GK-dimension one}
\label{xxsec1.4}

After the progress described in Subsection \ref{xxsec1.3} it is 
natural to ask if one can achieve a classification extending that 
of Theorems \ref{xxthm1.10}-\ref{xxthm1.12} by removing one or 
more of the hypotheses that $(A)$ $\Bbbk$ has characteristic 0, 
$(B)$ $H$ is AS-regular or $(C)$ $H$ is prime. We briefly consider 
each of these in turn.

\bigskip
\noindent 
$(A)$ This may be the easiest of the three, but it has received 
minimal attention so far. For the sake of precision, we formulate 
a specific question:

\begin{question}
\label{xxque1.13} 
Let $\Bbbk$ have characteristic $p > 0$. Classify the prime affine 
AS-regular Hopf $\Bbbk$-algebras $H$ of GK-dimension 1 when 
(i) $p \nmid io(H)$ and (ii) $p | io(H)$.
\end{question}

\noindent 
It may be that - at least in case (i) - this requires nothing 
more than a careful check of the arguments used for the results 
of Subsection \ref{xxsec1.3}. Along similar lines, one should also 
consider what happens when $\Bbbk$ is not algebraically closed. 
Even in the commutative case, new examples occur, as noted in 
\cite[Example 8.3]{LWZ07}.
\bigskip

\noindent 
$(B)$ Consider now the AS-regularity hypothesis. it was known 
already from \cite[Example 7.3]{BZ10} that even in characteristic 
0 there exist prime affine Hopf algebras of GK-dimension 1 which 
are not regular. Liu \cite{Li20} started a program to classify 
all prime affine non-regular Hopf algebras $H$ of GK-dimension 
one when $\Bbbk$ has characteristic 0. By \cite{WZ03}, despite 
the absence of regularity, such an algebra $H$ is still 
AS-Gorenstein, so homological integrals are defined and homological 
invariants such as the integral order and integral minor can 
continue to be used. Liu's strategy is to impose two hypotheses 
on $H$ which are known \emph{consequences} of regularity. For 
example, one of the hypotheses is to assume the existence of a 
nontrivial one-dimensional H-module $M$, specifically a module 
$M$ such that $o(M) = \mathrm{PIdeg}H$, where $o(M)$ is as 
defined in Subsection \ref{xxsec1.2}. Assuming these two extra 
hypotheses Liu is able to obtain a beautiful classification in 
which many new Hopf algebras occur, these being what he calls 
\emph{fractional} versions of some of the regular algebras. This 
leaves an obvious project:

\begin{question}
\label{xxque1.14} 
Classify prime affine Hopf $\Bbbk$-algebras of GK-dimension 1 in 
characteristic 0 without assuming Liu's two supplementary 
hypotheses.
\end{question}

\noindent $(C)$ 
What happens when the primeness hypothesis is omitted? Recall 
that an affine algebra of GK-dimension one satisfies a polynomial 
identity (PI), thanks to \cite{SSW85}. By \cite[Lemma 5.3(b)]{LWZ07}, 
every noetherian affine AS-regular Hopf algebra satisfying a PI 
is a direct sum of prime rings. In view of this and considering 
known examples, Lu, Wu and Zhang propose the following conjectural 
structure of a noetherian affine AS-regular Hopf $\Bbbk$-algebra 
$H$ of GK-dimension one. Such an algebra $H$ should fit into a 
short exact sequence
\begin{equation}
\label{E1.14.1}\tag{E1.14.1}
0\to H_{dis} \to H\to H_{conn}\to 0
\end{equation}
where the {\it connected component}  $H_{conn}$ is an AS-regular 
prime factor Hopf algebra of $H$ of GK-dimension 1, and the 
{\it discrete component} $H_{dis}$ should be a finite dimensional 
subalgebra of $H$ with a braided Hopf algebra structure. In 
\cite[Theorem 6.5(b)]{LWZ07}, Lu, Wu and Zhang were able to prove 
much of this, but \emph{not} the finite dimensionality of $H_{dis}$. 
We therefore ask:

\begin{question}
\label{xxque1.15} Is the above conjecture correct?
\end{question}

\noindent The answer to this question might well be important 
for the structure of AS-regular noetherian Hopf algebras which 
are \emph{not} of GK-dimension 1, nor even PI.

In parallel with trying to answer Question \ref{xxque1.15}, one 
should try to construct all possible algebras provided by the 
recipe of Lu, Wu and Zhang. This philosophy motivates the following 
question.

\begin{question}[Gongxiang Liu \cite{Li20a}]
\label{xxque1.16}  
Let $\Bbbk$ have characteristic 0 and let $H$ be an affine 
prime AS-regular Hopf algebra of GK-dimension 1, as listed
in Lemma \ref{xxlem1.1} and Theorems \ref{xxthm1.10}, 
\ref{xxthm1.11} and \ref{xxthm1.12}. 
Classify finite dimensional semisimple braided Hopf algebras 
in the Yetter-Drinfeld category $^{H}_{H} {\mathcal {YD}}$.
These objects should play the role of $H_{dis}$ in 
\eqref{E1.14.1}.
\end{question}

\subsection{The bridge from GK-dimension 1 to GK-dimension 2}
\label{xxsec1.5} 
As a step towards looking at prime affine Hopf algebras of 
GK-dimension 2 in the next section, we first note that, when 
$\Bbbk$ has characteristic 0, many such algebras can be 
constructed as extensions of an (in general braided) Hopf 
algebra of GK-dimension 1 by a second such algebra. This is 
true in the commutative case, as is recalled in Subsection 
\ref{xxsec2.1} below, and is also true in the cocommutative 
case, as can be deduced from the structure theorem of 
Cartier-Kostant-Gabri$\grave{\mathrm{e}}$l \cite[5.6.4, 5.6.5]{Mo93}. 
Many more examples can be constructed from the algebras 
featuring in Subsection \ref{xxsec1.3} by use of the tensor 
product, as in part (3) of the following result.

\begin{proposition}
\label{xxpro1.17} 
Let $\Bbbk$ be an algebraically closed field of 
characteristic 0, and let $H$ and $K$ be prime affine 
AS-regular Hopf $\Bbbk$-algebras of GK-dimension one. 
Let $n$ and $m$ be the integral orders of $H$ and $K$ 
respectively.
\begin{enumerate}
\item[(1)]
The center $Z(H)$ of $H$ is isomorphic to either $\Bbbk[t]$ 
or $\Bbbk[t^{\pm 1}]$.
\item[(2)]
The Goldie quotient ring of $H$ is isomorphic to $M_n(\Bbbk(t))$.
\item[(3)] 
The tensor product $H\otimes K$ is a noetherian prime affine 
AS-regular Hopf algebra of GK-dimension 2. Moreover $H \otimes K$ 
is a free module of rank $nm$ over its center $Z(H) \otimes Z(K)$.
\end{enumerate}
\end{proposition}

\begin{proof}
(1) This follows from the classification in Subsection \ref{xxsec1.3}. 
One can check the assertion directly for the algebras in Theorems 
\ref{xxthm1.9}, \ref{xxthm1.10} and \ref{xxthm1.11}(1). For the 
generalized Liu algebra $B(n,w,\xi)$ in Theorem \ref{xxthm1.11}(2), 
the assertion follows from \cite[Theorem 3.4(b)]{BZ10}. For the 
algebras $D := D(m,d,\xi)$ in Theorem \ref{xxthm1.12}, we claim 
that $Z(D) = \Bbbk[t]$ where $t=x+x^{-1}$. We outline the argument, 
freely using the notation introduced in \cite{WLD16}. Indeed, the 
algebra $D$ has a bigrading $D=\bigoplus D_{ij}$ with details given 
in \cite[equation (4.7)]{WLD16}. From this bigrading it is not hard 
to see that $Z(D) \subseteq D_{00}$, where $D_{00}$ is exactly the 
Hopf subalgebra $\Bbbk[x^{\pm 1}]$. Since $xu_i=u_ix^{-1}$ for 
$0\leq i\leq m-1$ by the definition of $D(m,d,\xi)$, the center is 
contained in the subalgebra $\Bbbk[x+x^{-1}]$. Conversely, it is 
easy to check that $x+x^{-1}$ commutes with all generators of $D$.

(2) This follows from part (1), by \cite{SW84} and Tsen's theorem 
\cite[p. 374]{Co77}.

(3) First, $H$ and $K$ are finite over their affine centers. 
By the K{\" u}nneth formula, we obtain that $H\otimes K$ has global
dimension at most 2 (a similar argument was used in the proof of 
\cite[Theorem 3.2]{LWZ09}). As a consequence, $H\otimes K$ is a 
noetherian affine PI AS-regular Hopf algebra of GK-dimension two. 
By \cite[Theorem A]{BG97}, $H\otimes K$ is a finite direct sum of 
prime rings.
 
We claim that $H\otimes K$ is prime. By (2), $Q(H)\cong M_n(\Bbbk(t))$ 
and $Q(K)\cong M_m(\Bbbk(s))$. Then 
$Q(H)\otimes Q(K)\cong M_{nm}(\Bbbk(s)\otimes \Bbbk(t))$, and hence
\begin{equation}
\label{E1.17.1}\tag{E1.17.1}
Q(H\otimes K) = Q(Q(H) \otimes Q(K)) \cong M_{nm}(\Bbbk(s,t)).
\end{equation} 
This forces $H\otimes K$ to be prime by Goldie's theorem. Next, 
using \eqref{E1.17.1},
\begin{equation} 
\label{E1.17.2}\tag{E1.17.2}
Z(H) \otimes Z(K) \subseteq Z(H \otimes K) \subseteq 
Z(Q(H\otimes K))  = \Bbbk (s,t) =Q(Z(H) \otimes Z(K)).
\end{equation}
Note that $Z(H \otimes K)$ is a finite $Z(H) \otimes Z(K)$-module 
and thus integral over $Z(H) \otimes Z(K)$. Since  the latter 
algebra is integrally closed we deduce from \eqref{E1.17.2} that 
$Z(H \otimes K) = Z(H) \otimes Z(K)$. Finally, the freeness of 
$H \otimes K$ over this (possibly partly Laurent) polynomial 
subalgebra follows from the Cohen-Macaulay property 
\cite[Theorems 0.1 and 0.2]{WZ03} and \cite[Theorem 3.7(ii)]{BM17}, 
and the rank from \eqref{E1.17.1} and \eqref{E1.17.2}.
\end{proof}

The following question, though at first glance rather 
technical, is suggested for the algebras $D(m,d,\xi)$ by 
considering the pattern followed by the other GK-dimension 
one families, and is highly relevant to the classification 
program for GK-dimension two Hopf domains discussed in 
Section 2.

\begin{question}
\label{xxque1.18}
Let $D(m,d,\xi)$ be the Hopf algebra in Theorem \ref{xxthm1.12}. 
Is there a noetherian affine Hopf domain $H$ of GK-dimension two 
containing either a grouplike element $g$ or a skew-primitive 
element $x$ with $H/\langle g-1 \rangle \cong D(m,d,\xi)$ or 
$H/\langle x \rangle \cong D(m,d,\xi)$ as Hopf algebras?
\end{question}

\noindent If such a Hopf algebra exists, it could \emph{not} 
be pointed (since $D(m,d,\xi)$ is not) and so would demonstrate 
the existence of non-pointed affine Hopf domains in dimension 2. At the 
moment, the project to classify noetherian affine Hopf 
domains of GK-dimension two only produces pointed Hopf algebras.

Note that the requirement that $H$ be a domain is key in 
Question \ref{xxque1.18}, since $H= D(m,d,\xi)[t]$ with $t$ 
primitive or $H= D(m,d,\xi)[t^{\pm 1}]$ with $t$ grouplike 
give easy prime examples as special cases of Proposition 
\ref{xxpro1.17}.

\section{Hopf algebras of GK-dimension 2}
\label{xxsec2}

In Section 2 we survey some results concerning Hopf 
algebra domains of GK-dimension two. We continue to 
assume that $\Bbbk$ is algebraically closed, 
and - except in Subsections \ref{xxsec2.3} and 
\ref{xxsec2.6} - we assume that $\Bbbk$ has characteristic 0.

\subsection{Commutative examples and dimension of tangent 
spaces}
\label{xxsec2.1}
The tensor product recipe of Proposition \ref{xxpro1.17} 
applied to the commutative Hopf algebras of Lemma 
\ref{xxlem1.1} yields the following three commutative Hopf 
domains of GK-dimension two:
\begin{equation}
\label{E2.1.0}\tag{E2.1.0}
{\mathcal O}(\Bbbk)^{\otimes 2},
{\mathcal O}(\Bbbk^{\times})^{\otimes 2},
{\mathcal O}(\Bbbk)\otimes 
{\mathcal O}(\Bbbk^{\times}).
\end{equation}

\noindent The remaining affine commutative Hopf 
$\Bbbk$-algebra domains of GK-dimension two arise as follows. 

\begin{example} \cite[Construction 1.1]{GZ10}
\label{xxex2.1} 
Let $n$ be a positive integer. Let $A(n,1)$ be the algebra 
generated by $x, x^{-1}, y$ subject to the relation $xy=yx$. 
There is a unique Hopf algebra structure on $A$ with 
comultiplication
$$\Delta(x)=x\otimes x, \quad \Delta(y)=y\otimes 1+x^n \otimes y.$$
\end{example}

\noindent Expressed in alternative language, examples 
\eqref{E2.1.0} and \ref{xxex2.1} give a list (which is of 
course well-known) of all connected two-dimensional affine 
algebraic groups over $\Bbbk$, as noted in \cite[Lemma 2.2]{GZ10}.

\medskip

In trying to classify all the affine Hopf domains of 
GK-dimension two it turns out to be both useful and natural 
to first impose an additional hypothesis. Namely we consider 
first such Hopf algebras $H$ for which 
\begin{equation}
\label{E2.1.1}\tag{E2.1.1}
\Ext^1_H(\Bbbk,\Bbbk)\neq 0.
\end{equation}
Recall that if $A$ is \emph{any} affine commutative 
$\Bbbk$-algebra of GK-dimension two and $\fm$ is a maximal 
ideal of $A$, then 
$$\dim_{\Bbbk} \fm/\fm^2=\dim_{\Bbbk} 
\Ext^1_{A}(A/\fm, A/\fm)\geq 2.$$
Indeed $\dim_{\Bbbk} \Ext^1_A(A/\fm,A/\fm)$ is equal to the 
dimension of the tangent space at the point $\fm$ in
${\text{Spec}}\; A$, so it seems reasonable to consider the 
class of quantum groups $H$ for which the tangent space at 
any point of $H$ is nontrivial. In fact, as we record in 
the next lemma, \eqref{E2.1.1} is equivalent to the condition 
that the corresponding quantum group contains a classical 
algebraic group of dimension 1; one might compare this with 
the analogous fact that every quantum projective plane in 
the sense of Artin and Schelter contains a classical 
commutative curve of dimension 1. 

\begin{lemma} \cite[Theorem 3.9]{GZ10}
\label{xxlem2.2}
Assume that $\Bbbk$ has characteristic 0 and let $H$ be a 
Hopf $\Bbbk$-algebra of GK-dimension at most two which is 
either affine or noetherian. Then $H$ satisfies 
\eqref{E2.1.1} if and only if $H$ has a Hopf quotient, 
denoted by $\overline{H}$, isomorphic to either 
$\Bbbk[t^{\pm 1}]$ with $t$ grouplike or $\Bbbk[t]$ with 
$t$ primitive.
\end{lemma}

\subsection{Noetherian Hopf domains of GK-dimension 2}
\label{xxsec2.2}
Following Lemma \ref{xxlem2.2} the classification of Hopf 
$\Bbbk$-algebras of GK-dimension 2 satisfying \eqref{E2.1.1} 
now splits into two cases: (i) $\overline{H}=\Bbbk[t^{\pm 1}]$
or (ii) $\overline{H}=\Bbbk[t]$. The following result is due 
to Goodearl and Zhang \cite{GZ10}. In it, families (I), (III) 
and (IV) constitute case (i), (II) and (V) case (ii).

\begin{theorem} \cite[Theorem 0.1]{GZ10}
\label{xxthm2.3} 
Let $\Bbbk$ be algebraically closed of characteristic 0, 
and let $H$ be a Hopf $\Bbbk$-algebra domain of GK-dimension 2 
satisfying \eqref{E2.1.1}. Then $H$ is noetherian if and only 
if $H$ is affine if and only if H is isomorphic to one of 
the following:
\begin{enumerate}
\item[(I)]
The group algebra $\Bbbk \Gamma$, where $\Gamma$ is either
\begin{enumerate}
\item[(Ia)] 
the free abelian group ${\mathbb Z}^2$, or
\item[(Ib)] 
the nontrivial semidirect product ${\mathbb Z}\rtimes {\mathbb Z}$.
\end{enumerate}
\item[(II)] 
The enveloping algebra $U({\mathfrak g})$, where 
${\mathfrak g}$ is either
\begin{enumerate}
\item[(IIa)] 
the 2-dimensional abelian Lie algebra over $\Bbbk$, or
\item[(IIb)] 
the Lie algebra over $\Bbbk$ with basis $\{x, y\}$ and 
$[x, y]=y$.
\end{enumerate}
\item[(III)] 
The Hopf algebras $A(n,q)$ in Example \ref{xxex2.4} below.
\item[(IV)] 
The Hopf algebras $B(n, p_0, \cdots, p_s, q)$ in Example 
\ref{xxex2.5} below.
\item[(V)] 
The Hopf algebras $C(n)$ in Example \ref{xxex2.6} below.
\end{enumerate}
\end{theorem}

\begin{example} \cite[Construction 1.1]{GZ10}
\label{xxex2.4}
Let $n$ be a positive integer and $q \in \Bbbk^{\times}$, 
and define $A(n,q)$ to be the algebra 
$\Bbbk \langle x^{\pm 1}, y \mid xy = q yx\rangle$.
There is a unique Hopf algebra structure on $A(n,q)$ under 
which $x$ is grouplike and $y$ is skew primitive, with
$\Delta(y) = y\otimes 1+ x^n \otimes y$. Clearly Example 
\ref{xxex2.1} is the case $q = 1$.
\end{example}

\begin{example} \cite[Construction 1.2]{GZ10}
\label{xxex2.5}
Let $n, p_0, p_1, \ldots , p_s$ be positive integers and 
$q \in \Bbbk^{\times}$ with the following properties:
\begin{enumerate}
\item[(a)] 
$s \geq 2$ and $1 < p_1 < p_2 <\cdots < p_s$;
\item[(b)]
$p_0 \mid n$ and $p_0, p_1,\ldots, p_s$ are pairwise 
relatively prime;
\item[(c)] 
$q$ is a primitive $\ell$th root of unity, where 
$\ell= (n/p_0)p_1p_2 \cdots p_s$.
\end{enumerate}

Set $m = p_1p_2 \cdots p_s$ and $m_i = \frac{m}{p_i}$ for 
$i = 1,\ldots, s$. Fix an indeterminate $y$ and consider the 
subalgebra $A := \Bbbk\langle y_1, \ldots, y_s\rangle 
\subseteq \Bbbk[y]$, where $y_i := y^{m_i}$ for $i = 1,\ldots, s$. 
The $\Bbbk$-algebra automorphism of $\Bbbk[y]$ sending $y \to qy$ 
restricts to an automorphism $\sigma$ of $A$. There is a unique 
Hopf algebra structure on the skew Laurent polynomial ring 
$B = A[x^{\pm 1};\sigma]$ such that $x$ is grouplike and the 
$y_i$ are skew primitive, with
$$\Delta (y_i) = y_i \otimes 1+ x^{m_i n} \otimes y_i$$ 
for $i = 1, \ldots, s$. This Hopf algebra is denoted by 
$B(n, p_0,\ldots , p_s, q)$.
\end{example}

\begin{example} \cite[Construction 1.4]{GZ10}
\label{xxex2.6}
Let $n$ be a positive integer, and set 
$C(n):=\Bbbk[y^{\pm 1}][x; (y^n - y) \frac{d}{dy}]$. That 
is, $C(n)$ is an Ore extension of $\Bbbk[y^{\pm 1}]$ with 
indeterminate $x$ satisfying the equation
$$ x f(y) =f(y) x+ (y^n-y) \frac{df(y)}{dy}$$
for all $f(y)\in \Bbbk[y^{\pm 1}]$. There is a unique Hopf 
algebra structure on $C(n)$ such that $y$ is grouplike and 
$x$ is skew primitive, with
$$\Delta(x) = x\otimes y^{n-1} +1\otimes x.$$ 
\end{example}

\bigskip

Naturally, in the light of Theorem \ref{xxthm2.3} Goodearl and 
Zhang asked the obvious question: does the classification remain 
true without assuming \eqref{E2.1.1}? In 2013 Wang, Zhang and 
Zhuang \cite{WZZ13} showed that the answer is ``No''. We'll 
review their result in Subsection \ref{xxsec2.5} 

Note that all the algebras listed in Theorem \ref{xxthm2.3} are 
assembled from two one-dimensional domains, as discussed in 
Subsection \ref{xxsec1.5}. In particular, $A(n,q)$ and $C(n)$ 
are Ore extensions of the Hopf algebra $\Bbbk[t^{\pm 1}]$, 
special cases of the Hopf Ore construction which we review in 
the next subsection.

\subsection{Hopf Ore extensions and IHOEs}
\label{xxsec2.3}
Hopf Ore extensions of a Hopf algebra $A$, denoted by 
$A[x;\sigma, \delta]$, were defined and studied by Panov 
\cite{Pan03} in 2003. His definition was refined and extended 
in \cite{BO15} and then in \cite{Hu19}, to 
\emph{iterated Hopf Ore extensions} (IHOEs) over a field 
$\Bbbk$. The following definition is significantly more 
general than the original one given in 
\cite[Definition 1.0]{Pan03} and its modification in 
\cite[Definition 2.1]{BO15}. The improvement here is due to 
recent work of Huang \cite{Hu19}. Recall that, given an 
algebra automorphism $\sigma$ of $A$, a $\sigma$-derivation 
$\delta$ of $A$ is a $\Bbbk$-linear endomorphism of $A$ such that 
$$\delta(ab) = \sigma(a)\delta(b) + \delta(a)b$$ 
for all $a,b \in A$.

\begin{definition}
\label{xxdef2.7} Let $\Bbbk$ be a field.
\begin{enumerate}
\item[(1)] 
Let $A$ be a Hopf $\Bbbk$-algebra. A {\it Hopf Ore extension} 
(abbreviated to {\it HOE}) of $A$ is an algebra $H$ such that
\begin{enumerate}
\item[(1a)]
$H$ is a Hopf $\Bbbk$-algebra with Hopf subalgebra $A$;
\item[(1b)]
there exist an algebra automorphism $\sigma$ of $A$ and a
$\sigma$-derivation $\delta$ of $A$ such that $H = A[x;\sigma,
\delta]$.
\end{enumerate}
\item[(2)] 
An {\it {\rm{(}}$n$-step{\rm{)}} iterated Hopf Ore extension of 
$\Bbbk$} {\rm{(}}abbreviated to {\it {\rm{(}}$n$-step{\rm{)}} 
IHOE {\rm{(}}of $\Bbbk${\rm{)}}}{\rm{)}} is a Hopf algebra 
\begin{equation} 
\notag
H = \Bbbk[X_1][X_2; \sigma_2 , \delta_2] \cdots 
[X_n; \sigma_n , \delta_n],
\end{equation} 
where
\begin{enumerate}
\item[(2a)]
$H$ is a Hopf algebra;
\item[(2b)]
$H_{(i)} := \Bbbk\langle X_1, \ldots , X_i \rangle$ is 
a Hopf subalgebra of $H$ for $i = 1, \ldots , n;$
\item[(2c)]
$\sigma_i$ is an algebra automorphism of $H_{(i-1)}$, 
and $\delta_i$ is a $\sigma_i$-derivation of $H_{(i-1)}$, 
for $i = 2, \ldots , n$.
\end{enumerate}
\end{enumerate}
\end{definition}

The definition of an HOE in \cite[Definition 2.1]{BO15} 
required, in addition to Definition \ref{xxdef2.7}(1a,1b), that
\begin{enumerate}
\item[(1c)]
there are $a,b\in A$ and $v, w\in A\otimes A$ such that
\begin{equation}
\notag
\Delta(x)=a\otimes x+x\otimes b+ v(x\otimes x)+ w.
\end{equation}
\end{enumerate}
However, by Huang's theorem  \cite[Theorem 1.3]{Hu19},
if $A \subseteq H$ are noetherian $\Bbbk$-algebras satisfying 
hypotheses (1a) and (1b) of Definition \ref{xxdef2.7}(1), and 
$A\otimes A$ is a domain, then, up to a change of variable,
\begin{enumerate}
\item[(1d)]
there are $a\in A$ and $w\in A\otimes A$ such that
\begin{equation}
\notag
\Delta(x)=a\otimes x+x\otimes 1+w.
\end{equation}
\end{enumerate}
This means that under the conditions in \cite[Theorem 1.3]{Hu19}, 
namely that $A \otimes A$ is a noetherian domain, Definition 
\ref{xxdef2.7}(1) is equivalent to \cite[Definition 2.1]{BO15}. 
For the subalgebras $H_{(i)}$ appearing in Definition 
\ref{xxdef2.7}(2) it is easy to show that 
$H_{(i)} \otimes H_{(i)}$ is  a noetherian domain, so a similar 
comment applies to the comultiplication of an IHOE over $\Bbbk$. There 
are moreover strong restrictions on the automorphisms $\sigma_i$ 
and derivations $\delta_i$ which can occur - for details, see 
\cite{BO15}, \cite{Hu19} or \cite{BZ20}.

In Examples \ref{xxex2.8} and \ref{xxex2.9} in the next 
subsection we construct Hopf Ore extensions that are 
generalizations of the Hopf Ore extensions $A(n,q)$ and $C(n)$ 
of Theorem \ref{xxthm2.3}. 

\subsection{Non-noetherian Hopf domains of GK-dimension 2}
\label{xxsec2.4}
In this subsection we describe two families of Hopf domains of 
GK-dimension two which are not affine and not 
noetherian. 

\begin{example}\cite[Example 1.2]{GZ17}
\label{xxex2.8}
Let $K = \Bbbk G$ where $G$ is a group, let $\chi: G \to 
\Bbbk^{\times}$ be a character of $G$ and choose any element 
$e$ in the center of $G$, denoted by $Z(G)$. Define an 
algebra automorphism $\sigma_{\chi}: K \to K$ by 
$$\sigma_{\chi}(g) = \chi(g)g, \quad \forall\; g\in G,$$
so $\sigma_{\chi}$ is the winding automorphism associated 
to the algebra map from $K$ to $\Bbbk$ induced by $\chi$. 
Let $\delta= 0$. By \cite[Example 5.4]{WZZ16}, 
$K[z;\sigma_{\chi}]$ is an HOE of $K$ with 
$\Delta(z)=z\otimes 1+e\otimes z$. This HOE is denoted by 
$A_G(e,\chi)$.

Now carry out the above construction when $G$ is a non-trivial 
non-cyclic subgroup of $({\mathbb Q},+)$. Then 
$$\GKdim \Bbbk G=1, \quad {\text{and}}\quad \GKdim A_G(e,\chi)=2,$$ 
using \cite[Lemma 2.2]{HK96} for the second equality. By 
\cite[Theorem 0.1(4)]{GZ17}, $A_G(e,\chi)$ satisfies 
\eqref{E2.1.1}. When, for instance, $\Bbbk={\mathbb C}$, there 
are many characters of $G$; for example, with 
$\lambda \in \mathbb{R}$, $\exp_{\lambda}: 
r\to \exp(2\pi i r \lambda)$ is a  character from 
$({\mathbb Q},+)$ to ${\mathbb C}^{\times}$. 
Since $A_G(e,\chi)$ is a free 
$\Bbbk G$-module it is clear that  $A_G(e,\chi)$ is not 
noetherian. Being a skew polynomial algebra with coefficient 
ring the group algebra of a torsion-free abelian group, 
$A_G(e,\chi)$ is a domain.
\end{example}

\begin{example} \cite[Example 1.3]{GZ17}
\label{xxex2.9} 
Again, let $K = \Bbbk G$ where $G$ is a group and fix 
$e \in Z(G).$ Let $\tau: G \to (\Bbbk,+)$ be an additive 
character of $G$. Define a $\Bbbk$-linear derivation 
$\delta: K \to K$ by 
$$\delta(g) = \tau(g)g(e-1),\quad \forall\;  g\in G.$$ 
Then $C_G(e, \tau) :=  K[z;\delta]$ is an HOE of $K$ when 
we define $\Delta(z)=z\otimes 1+e\otimes z$.

As in the previous example, let $G$ be a nontrivial non-cyclic 
subgroup of $({\mathbb Q},+)$. Again by \cite[Lemma 2.2]{HK96} 
(or \cite[Theorem 1.2(1)]{Zha97})
$$\GKdim \Bbbk G=1, \quad {\text{and}}\quad \GKdim C_G(e, \tau)=2.$$ 
Since we are assuming $\Bbbk$ has characteristic zero, 
there are many additive characters from $G$ to $(\Bbbk,+)$. For 
example, let $\lambda \in \Bbbk$; then 
$i_{\lambda}: r\to r \lambda$ is 
an additive character from $G\to (\Bbbk,+)$.

Denote the augmentation ideal of $K$ by $K^+$. Since 
$C_G(e, \tau)K^+$ is an ideal of $C_G(e, \tau)$ with 
$C_G(e, \tau)/C_G(e, \tau)K^+ \cong \Bbbk[z]$,
$C_G(e,\chi)$ satisfies \eqref{E2.1.1}. In the same 
way as Example \ref{xxex2.8}, $C_G(e, \tau)$ is not 
noetherian, but it is a domain.
\end{example}

The next family of examples are \emph{not} HOEs, 
although they are like the others assembled from 
two domains of GK-dimension 1. Since the details 
are quite technical, we only give a brief sketch here, 
referring the reader to \cite{GZ17} for more information. 

\begin{example} \cite[Construction 3.1]{GZ17}
\label{xxex2.10} 
Let $G$ be a nontrivial non-cyclic subgroup of 
$({\mathbb Q},+)$, let $\{p_i : i \in I\}$ be an infinite 
set of pairwise relatively prime integers, and let $\chi$ 
be a character of a certain group $GM$ with 
$G \subseteq GM \subseteq ({\mathbb Q},+)$. Given this 
data one can construct (over several pages) a non-affine 
Hopf domain of GK-dimension two, denoted by 
$B_G(\{p_i\},\chi)$. By \cite[Construction 1.2]{GZ17},
$B_G(\{p_i\},\chi)$ satisfies \eqref{E2.1.1} and 
is not noetherian.
\end{example}

Here is the classification result parallel to Theorem
\ref{xxthm2.3}. Let ${\mathbb Z}_{\langle 2\rangle}$ 
denote the localization of the ring ${\mathbb Z}$ at 
the maximal ideal $\langle 2 \rangle $.

\begin{theorem}\cite[Theorem 0.1]{GZ17}
\label{xxthm2.11} 
Assume that $\Bbbk$ has characteristic 0 and let $H$ 
be a Hopf domain over $\Bbbk$ of GK-dimension two which 
is not noetherian and satisfies \eqref{E2.1.1}. Then $H$ 
is isomorphic as a Hopf algebra to one of the following.
\begin{enumerate}
\item[(1)]
$\Bbbk G$ where $G$ is a non-finitely generated subgroup 
of ${\mathbb Q}^{2}$ containing ${\mathbb Z}^2$.
\item[(2)]
$\Bbbk G$ where $G = L\rtimes R$ for subgroups $L$ of 
${\mathbb Q}$ containing ${\mathbb Z}$ and $R$ of 
${\mathbb Z}_{(2)}$ containing ${\mathbb Z}$, with at 
least one of $L$ and $R$ not finitely generated.
\item[(3)] 
$A_{G}(e,\chi)$ in Example \ref{xxex2.8}.
\item[(4)]
$C_{G}(e,\tau)$ in Example \ref{xxex2.9}.
\item[(5)]
$B_G(\{p_i\},\chi)$ in Example \ref{xxex2.10}.
\end{enumerate}
\end{theorem}

Thus Theorems \ref{xxthm2.3} and \ref{xxthm2.11} together 
with Lemma \ref{xxlem1.1} and Remark \ref{xxrem1.2}(1) give 
a complete classification when $\Bbbk$ has characteristic 0 
of the Hopf $\Bbbk$-algebra domains of GK-dimension at most 
2 which satisfy \eqref{E2.1.1}. But not all Hopf domains of 
GK-dimension 2 satisfy \eqref{E2.1.1}, as we explain in the 
next subsection.

\subsection{Hopf domains of GK-dimension 2 with no infinite 
classical subgroup}
\label{xxsec2.5} 
The hope that the algebras of Subsection \ref{xxsec2.2} 
constitute a complete list of affine Hopf domains of 
GK-dimension 2 is demolished by the following result. 
As the notation suggests, the new algebras are variants 
of those in Examples \ref{xxex2.5}.

\begin{theorem} \cite{WZZ13}
\label{xxthm2.12} 
Assume that $\Bbbk$ has characteristic 0. Let 
$s, p_1, \ldots , p_s$ be integers, all greater than 1 
and with $\mathrm{gcd}(p_i,p_j) = 1$ when $i \neq j$. 
Fix scalars $\sigma_1, \ldots , \sigma_s$ in $\Bbbk$, 
not all equal. For $\ell := \Pi_{i=1}^s p_i$, choose 
$q$, an $\ell$th root of 1 in $\Bbbk$. Then there is a 
family $B(n,\{p_i\}_1^s,q,\{\sigma_i\}_1^s)$ of Hopf 
$\Bbbk$-algebras with the following properties.
\begin{enumerate}
\item[(1)] 
$B(n,\{p_i\}_1^s,q,\{\sigma_i\}_1^s)$ is an affine domain 
of GK-dimension 2.
\item[(2)] 
$B(n,\{p_i\}_1^s,q,\{\sigma_i\}_1^s)$ does not satisfy 
\eqref{E2.1.1}.
\item[(3)] 
$B(n,\{p_i\}_1^s,q,\{\sigma_i\}_1^s)$ is pointed, generated 
by grouplikes and skew primitives.
\item[(4)] 
$B(n,\{p_i\}_1^s,q,\{\sigma_i\}_1^s)$ is a finite module 
over its affine center, hence is noetherian and satisfies 
a polynomial identity.
\item[(5)] 
$\mathrm{gldim}B(n,\{p_i\}_1^s,q,\{\sigma_i\}_1^s)  < 
\infty$ if and only if $s=2$ and $\sigma_1 \neq \sigma_2$. 
In this case the global dimension is 2.
\item[(6)] 
Up to a permutation of $\{1, \ldots , s \}$ and 
multiplication of $\{\sigma_i\}_1^s$ by a non-zero scalar, 
these Hopf algebras are mutually non-isomorphic.
\end{enumerate}
\end{theorem}

Wang, Zhang and Zhuang show in \cite{WZZ13} that any as 
yet unknown example of an affine Hopf domain of 
GK-dimension 2 in characteristic 0, especially if 
generated by grouplikes and skew primitives, would 
have to exhibit very constrained features. They therefore ask:

\begin{question} \cite[Remark 4.5]{WZZ13}
\label{xxque2.13}  
Assume $\Bbbk$ has characteristic 0. Are the algebras of 
Theorem \ref{xxthm2.12} the only affine Hopf $\Bbbk$-algebra 
domains of GK-dimension 2 not satisfying \eqref{E2.1.1}?
\end{question}

It is tempting, given the classification in characteristic 
0 of prime AS-regular algebras of dimension 1 and the 
partial classification of GK-dimension 2 Hopf domains, 
to consider prime AS-regular Hopf algebras of GK-dimension 2. 
Proposition \ref{xxpro1.17} indicates that this may be a 
formidable task, so perhaps a preliminary step might be 
the following.

\begin{question} [Dingguo Wang \cite{Wa20}]
\label{xxque2.14}
Find homological invariants (additional to, or 
generalising, the integral order and integral minor) which 
are defined for pointed Hopf algebras, and which may help 
to classify prime pointed Hopf algebras of GK-dimension two.
\end{question}

\subsection{2-step IHOEs in positive characteristic}
\label{xxsec2.6}
In this subsection let $p$ be a prime and let $\Bbbk$ be an 
algebraically closed field of characteristic $p$. Consider 
the project to classify the affine Hopf $\Bbbk$-algebra 
domains of GK-dimension at most 2. If $H$ is such an algebra 
with $\mathrm{GKdim} H = 1$, then $H$ is commutative by 
Tsen's theorem \cite[p. 374]{Co77}, and so $H= \Bbbk [t]$ or 
$\Bbbk [t^{\pm 1}]$ by Lemma \ref{xxlem1.1}.

However when $\mathrm{GKdim}H = 2$ the situation appears 
much more complex than in characteristic 0.  The 
\emph{connected} Hopf $\Bbbk$-algebras, that is those 
whose coradical is $\Bbbk$, form a good starting point. In 
arbitrary characteristic every IHOE is connected by 
\cite[Proposition 2.5]{BO15}. The cocommutative connected 
Hopf algebras in characteristic 0 are the enveloping 
algebras of Lie algebras. Zhuang showed in 
\cite[Proposition 7.6]{Zhu13} that in characteristic 0, there 
are only two connected Hopf algebras of GK-dimension 2, and 
both are cocommutative IHOEs: namely the enveloping algebras 
of the two 2-dimensional Lie algebras. But in positive 
characteristic it's a different story, as we now explain. 

Let ${\bf d_s} = \{d_s\}_{s\geq 0}$, ${\bf b_s} = 
\{b_s\}_{s\geq 0}$ and 
${\bf c_{s,t}} = \{c_{s,t}\}_{0\leq s<t}$ be sequences 
of scalars in $\Bbbk$ with only finitely many nonzero 
terms. Let $H({\bf d_s,b_s,c_{s,t}})$ denote the 
Ore extension $\Bbbk[X_1][X_2;\mathrm{Id},\delta]$ 
where 
\begin{equation}
\notag
\delta(X_1)=\sum_{s\geq 0} d_s X_1^{p^s},
\end{equation}
and where the coalgebra structure is as follows. Define 
maps $\Delta$, $\epsilon$, and $S$ on $\{X_1, X_2\}$ by 
setting
$$\begin{aligned}
\Delta(X_1)&=X_1\otimes 1+1\otimes X_1,\\
\Delta(X_2)&=X_2\otimes 1+1\otimes X_2+ w,\\
\epsilon(X_1)&=\epsilon(X_2)=0,\\
S(X_1)&=-X_1,\\
S(X_2)&=-X_2-m(\mathrm{Id}\otimes S)(w),
\end{aligned}
$$ 
where 
\begin{align}
\notag
w&=\sum_{s\geq 0} b_s 
  \left(\sum_{i=1}^{p-1} \frac{(p-1)!}{i! (p-i)!} (X_1^{p^s})^{i}
  \otimes (X_1^{p^s})^{p-i}\right)\\
	\notag
&\qquad \qquad+\sum_{0 \leq s<t} c_{s,t} \left(X_1^{p^s}\otimes 
  X_1^{p^t}- X_1^{p^t}\otimes X_1^{p^s}\right).
\end{align}
  
\begin{theorem}\cite[Theorem 0.3(3)]{BZ20}
\label{xxthm2.15}  
Let $\Bbbk$ be algebraically closed of characteristic 
$p > 0$ and let $H$ be a 2-step IHOE over $\Bbbk$. 
\begin{enumerate}
\item[(1)] 
The Hopf algebra $H$ is isomorphic to 
$H({\bf d_s,b_s,c_{s,t}})$ for some choice of scalars
$\{d_s\}_{s\geq 0}$, $\{b_s\}_{s\geq 0}$ and 
$\{c_{s,t}\}_{0\leq s<t}$.
\item[(2)] 
Two such Hopf algebras $H({\bf d_s,b_s,c_{s,t}})$ and 
$H(({\bf d'_s,b'_s,c'_{s,t}})$ are isomorphic as Hopf 
algebras if and only if there are nonzero scalars 
$\alpha, \beta \in \Bbbk$ such that
$$
d'_s=d_s \alpha^{p^s-1}\beta^{-1}, \quad 
b'_s=b_s \alpha^{p^{s+1}}\beta^{-1}, \quad
c'_{s,t}=c_{s,t} \alpha^{p^s+p^t}\beta^{-1}
$$
for all $s,t$.
\end{enumerate}
\end{theorem}

\noindent The connected unipotent $\Bbbk$-groups $U$ of 
dimension $n$ are precisely those algebraic groups whose 
coordinate ring is an $n$-step commutative IHOE over 
$\Bbbk$, as explained in \cite[$\S$6.1]{BZ20}. Thus Theorem 
\ref{xxthm2.15} incorporates the classification of the 
2-dimensional connected unipotent groups over 
$\Bbbk$ - namely, their coordinate rings are the algebras 
$H({\bf 0,b_s,c_{s,t}})$. The classification in terms of 
the groups rather than their coordinate rings can be found 
in \cite[$\S$3.7]{KMT74}. For further information 
concerning the algebras  $H({\bf d_s,b_s,c_{s,t}})$, 
see \cite{BZ20}.

As suggested above, one should view Theorem \ref{xxthm2.15} 
as the first step in a program to classify the connected 
affine Hopf domains of GK-dimension 2 over $\Bbbk$ when 
$\Bbbk$ has positive characteristic. From this perspective 
and in the light of Zhuang's classification in 
characteristic 0, \cite[Proposition 7.6]{Zhu13}, it is natural 
to ask:

\begin{question} \cite[Question 6.3]{BZ20}
\label{xxque2.16} 
Is every affine connected Hopf domain over $\Bbbk$ of 
GK-dimension 2 an IHOE?
\end{question}

\section{Further comments, projects and open questions}
\label{xxsec3} 
In this final section $\Bbbk$ denotes an arbitrary 
algebraically closed field unless otherwise stated. We 
review the questions scattered throughout the previous 
sections and gather them into a few overarching projects, 
and then suggest a couple of further directions of research. 
All these questions and projects should be viewed as 
supplementing (and in some cases repeating) questions in 
other survey articles, for example \cite{And14, And17, Br98, 
Br07, BGi14, BGLZ14, Go13, KWZ19}.

First we note the familiar questions reappearing in this 
paper which concern the entire category of affine Hopf 
algebras of finite GK-dimension, namely Question 
\ref{xxque0.1} and Conjecture \ref{xxcor0.2}. To these 
we add the following, which has been implicit at several 
points in the foregoing pages.

\begin{question}
\label{xxque3.1} 
Is every affine Hopf algebra of finite GK-dimension noetherian?
\end{question}

The only evidence in favor of a positive answer to Question 
\ref{xxque3.1} seems to be the absence of counterexamples. The 
converse is false - consider for example the group algebra 
of a polycyclic group with is not nilpotent-by-finite. It 
does however seem to be unknown whether a noetherian Hopf 
algebra has to be affine.

\subsection{Projects in GK-dimensions one and two}
\label{xxsec3.1} 
The questions listed earlier concerning Hopf algebras of 
GK-dimension 1 are Questions \ref{xxque1.13}, \ref{xxque1.14}, 
\ref{xxque1.15}, and \ref{xxque1.16}. These should be 
regarded as key components on the road towards:  

\begin{project}
\label{xxpro3.2}
Classify semiprime noetherian Hopf $\Bbbk$-algebras of 
GK-dimension one.
\end{project}

\noindent Question \ref{xxque1.18} focuses on the 
possibility of a close link between GK-dimensions one and 
two. It would be answered as an immediate corollary of a 
classification result in GK-dimension two.

The target of classifying all affine or noetherian semiprime 
(or even prime) Hopf algebras of GK-dimension two seems 
rather too ambitious at present, so we propose:

\begin{project}
\label{xxpro3.3}
Classify Hopf $\Bbbk$-algebra domains of GK-dimension two.
\end{project}

Additional hypotheses to impose as milestones on the route 
to Project \ref{xxpro3.3} include those of Question 
\ref{xxque2.13} - namely that $\Bbbk$ has characteristic 
0 and the algebra is affine; and those of Question 
\ref{xxque2.16}, that $\Bbbk$ has positive characteristic 
and the Hopf algebra is affine and connected.

\subsection{Beyond GK-dimension two}
\label{xxsec3.2}
The following special case of Question \ref{xxque0.1} is 
currently open:

\begin{question}
\label{xxpro3.4} 
Does there exist a Hopf $\Bbbk$-algebra $H$ with 
$2 < \mathrm{GKdim}H < 3$?
\end{question}

\noindent Strong evidence in support of a positive answer 
to this question is given by \cite[Theorem 0.1]{WZZ16}, 
which confirms that the answer is "No'' when $\Bbbk$ has 
characteristic 0, with $H$ affine and pointed. In the light 
of their result Wang, Zhang and Zhuang proposed the following, 
suggesting that double Ore extensions may be an important 
tool for this purpose:

\begin{project} \cite[Remark 7.5]{WZZ16}
\label{xxpro3.5} 
In characteristic 0, classify affine pointed Hopf domains 
of GK-dimension three. 
\end{project}

\subsection{Connected (graded) Hopf algebras}
\label{xxsec3.3} 
Observe that if the pointed hypothesis in Project \ref{xxpro3.5} 
is strengthened to require that the Hopf algebra is connected, 
the proposed classification has been achieved by Zhuang 
\cite{Zhu13}, and even in dimension four in \cite{WZZ15}. As 
these results suggest, the class of connected Hopf algebras, 
and even the smaller classes of those which are graded as 
algebras and/or as coalgebras, form interesting testing grounds. 

For example, in \cite{BGZ19}, Brown, Gilmartin and Zhang studied, 
for $\Bbbk$ of characteristic 0,  Hopf $\Bbbk$-algebras of 
finite GK-dimension that are connected graded as an algebra.
In particular, they proved that such a Hopf algebra is a 
noetherian domain which is GK-Cohen-Macaulay, AS regular and 
Auslander regular of global dimension equal to its GK-dimension, 
and is Calabi-Yau. As a consequence, it satisfies Poincar{\' e} 
duality.

Remaining in characteristic 0, but imposing the stronger 
hypothesis on the Hopf algebra $H$ of finite GK-dimension, 
that it is connected graded both as an algebra \emph{and} 
as a coalgebra,  Zhou, Shen and Lu proved the striking 
result \cite{ZSL19} that $H$ is an IHOE \cite[Theorem B]{ZSL19}. 
Naturally, this leads to

\begin{question}\cite[Question C]{ZSL19}
\label{xxpro3.6} 
With $\Bbbk$ of characteristic 0, suppose the Hopf 
$\Bbbk$-algebra $H$ has finite GK-dimension and is 
connected graded as an algebra. Is H an IHOE?
\end{question}

\subsection*{Acknowledgments}
\label{xxsec5}
The authors thank the editors for the invitation to provide this 
submission, and thank Professors Gongxiang Liu and Dingguo Wang 
for several conversations on the subject, for reading an early 
draft of the paper, for suggesting some projects. K.A. Brown was 
partially supported by Leverhulme Emeritus Fellowship 
(EM-2017-081-9) and J.J. Zhang by the US National Science 
Foundation (No. DMS-1700825).

\end{document}